\documentclass[reqno,oneside,11pt]{amsart}

\usepackage{geometry}
\usepackage{graphicx}
\usepackage{verbatim}
\usepackage{enumitem}
\usepackage{mathrsfs}
\usepackage[all]{xy}
\usepackage{xcolor}
\colorlet{darkishRed}{red!80!black}
\colorlet{darkishBlue}{blue!60!black}
\colorlet{darkishGreen}{green!60!black}
\usepackage{hyperref}
\hypersetup{
    draft = false,
    bookmarksopen=true,
    colorlinks,
    linkcolor={red!60!black},
    citecolor={green!60!black},
    urlcolor={blue!60!black}
}
\usepackage[nameinlink, capitalise, noabbrev]{cleveref}
\crefformat{enumi}{#2#1#3}
\crefformat{equation}{#2(#1)#3}
\newtheorem{theorem}{Theorem}[section]

\newtheorem{lemma}[theorem]{Lemma}

\newtheorem{mainresult}{Theorem} 
\newtheorem{openm}[mainresult]{Open question}

\theoremstyle{definition}
\newtheorem{example}[theorem]{Example}

\theoremstyle{remark}

\newcommand{\se}{\subseteq}

\usepackage{subfig}

\title[On the edge-chromatic number of 2-complexes]{On the edge-chromatic number of 2-complexes\\\medskip --- Note ---}

\author{Jan Kurkofka${}^\clubsuit$}
\thanks{${}^\clubsuit$funded by EPSRC, grant number EP/T016221/1}
\author{Emily Nevinson${}^\diamondsuit$}
\thanks{${}^\diamondsuit$funded by EPSRC, CDT in Topological Design EP/S02297X/1}
\address{University of Birmingham, Birmingham, UK}


\newcommand{\sm}{\setminus}

\newcommand{\Sp}{\mathbb{S}}

\newcommand{\R}{\mathbb{R}}

\begin{document}
\begin{abstract}
We propose an open question that seeks to generalise the Four Colour Theorem from two to three dimensions.
As an appetiser, we show that 12 instead of four colours are both sufficient and necessary to colour every 2-complex that embeds in a prescribed 3-manifold.
However, our example of a 2-complex that requires 12 colours is not simplicial.
\end{abstract}

\maketitle

\section*{Introduction}

Motivated by the Four Colour Theorem, we raise the following open question, which essentially seeks to generalise the Four Colour Theorem from two to three dimensions.
An (\emph{edge-})\emph{colouring} of a 2-complex $C$ assigns to every edge of $C$ a colour such that two edges $e$ and $e'$ receive different colours whenever $e$ and $e'$ share an endvertex~$v$ and the boundary of some 2-cell of~$C$ enters and leaves $v$ through $e$ and~$e'$, respectively.

\begin{openm}
Let $M$ be a 3-manifold.
What is the least integer~$k$ such that every simplicial 2-complex that embeds in $M$ is $k$-colourable?
\end{openm}

In this note, we show that the answer is `$k=12$' for every 3-manifold $M$ if `simplicial' is dropped from the question:

\begin{mainresult}\label{main}\,
\begin{enumerate}[label={\textnormal{(\arabic*)}}]
\item\label{main:sufficient} Every 2-complex that embeds in a 3-manifold is 12-colourable.
\item\label{main:necessary} There is a 2-complex that embeds in $\R^3$ and which is not 11-colourable.
\end{enumerate}
\end{mainresult}

This note is part of a project that aims to extend planar graph theory to three dimensions.
Previously, the following results have been extended:

\begin{tabular}{l l c c}
& & \emph{2D} & \emph{3D}\\
$\bullet$ & Kuratowski's Theorem & \cite{Kura,diestel2016graph} & \cite{JCseries}\\
$\bullet$ & The excluded-minors characterisation of outerplanar graphs & folklore & \cite{OS}\\
$\bullet$ & Heawood’s Theorem on colourings of plane triangulations & \cite{evenhw} & \cite{HW,Joswig,tim}\\
$\bullet$ & Whitney's Theorem on unique drawings of 3-connected graphs & \cite{WhitneyUnique,diestel2016graph} & \cite{WhitneyThree}
\end{tabular}

\section{Terminology}

We use the terminology of~\cite{diestel2016graph}.
In this note, graphs may have loops and parallel edges.

\subsection{1-complexes}
Let $G$ be a graph with vertex-set $V$ and edge-set~$E$.
We can obtain a topological space from~$G$, called the \emph{1-complex} of~$G$ and also denoted by~$G$, as follows.
The \emph{0-skeleton} of $G$ is $V$ equipped with the discrete topology.
For every edge $e\in E$, let $[0,1]_e$ be a copy of the unit interval, disjoint from $V$ and from all other copies $[0,1]_{e'}$.
Furthermore, arbitrarily fix a map $\varphi_e\colon\{0_e,1_e\}\to V$ such that the image of $\varphi_e$ is equal to the set of ends of~$e$ (so there are two choices for $\varphi_e$ if $e$ is not a loop, and only one choice if $e$ is a loop).
The \emph{1-complex} of $G$ is obtained from the 0-skeleton of~$G$ by adding all copies $[0,1]_e$ for all edges $e\in E$ and identifying $0_e$ and $1_e$ with their images under~$\varphi_e$.
Note that taking the quotient as above also defines a topology on the 1-complex.
For convenience, we now change the notation $[0,1]_e$ to refer to $[0,1]_e$ after taking the quotient as above, so that we have~$[0,1]_e\se G$.
We then call $[0,1]_e$ a \emph{topological edge} of the 1-complex~$G$, and write $e$ for $[0,1]_e$ when there is no danger of confusion.
The \emph{third-edges} of~$G$ are the closed intervals $[0,\frac{1}{3}]_e$ and $[\frac{2}{3},1]_e$ of the topological edges~$[0,1]_e$, where $e$ ranges over all edges of the graph~$G$.

\subsection{2-complexes}

A \emph{2-complex} $C$ is a topological space obtained from a 1-complex~$G$ by disjointly adding closed 2-dimensional discs $D_i$ ($i\in I$), fixing a continuous \emph{gluing map} $\varphi_i\colon\partial D_i\to G$ for each~$i$, and identifying $x$ with $\varphi_i(x)$ for all $i$ and $x\in\partial D_i$.
In this note, we will only need to consider 2-complexes whose gluing maps $\varphi_i$ follow closed walks in~$G$ at constant nonzero speed.
This will allow us to also view the gluing maps $\varphi_i$ from a combinatorial perspective, through the closed walks they \emph{correspond} to.
The subspaces of~$C$ obtained from the discs $D_i$ by gluing their boundaries to the 1-skeleton are the \emph{2-cells} of~$C$.
The vertices and edges of~$C$ are the vertices and edges of its 1-skeleton.
A 2-complex $C$ is said to be \emph{simplicial} if $G$ is simple and each gluing map follows a closed walk that goes once around a triangle.

1-complexes and 2-complexes are instances of the more general cell complexes, see~\cite{Hatcher}.

\subsection{Link graphs}
Let $C$ be a 2-complex with 1-skeleton~$G$ and gluing maps~$\varphi_i$ ($i\in I$) for its 2-cells.
The \emph{link graph} of~$C$, which we denote by~$L(C)$, is defined as follows.
The vertices of~$L(C)$ are the third-edges of~$G$.
For each~$i$, we follow $\varphi_i$ along the circle that is its domain (the direction does not matter), and we add an edge between two vertices $I$ and $J$ in $L(C)$ whenever $I$ and $J$ share a vertex $v$ in~$G$ and $\varphi_i$ first traverses $I$ to reach $v$ and then traverses~$J$ next (or vice versa).
Hence the link graph $L(C)$ may contain parallel edges and loops, even if $C$ only has one 2-cell.

A \emph{pairing} of a set~$S$ is partition of~$S$ into classes of size two.
A \emph{paired graph} is a pair of a graph $G$ and a pairing $\pi$ of its vertex set.
Every link graph has a \emph{default pairing} in which every two third-edges that are included in the same topological edge form a class.
When we view a link graph as a paired graph, we always use the default pairing.

\subsection{Colourings of paired graphs and 2-complexes}

A \emph{pair-colouring} of a paired graph $(G,\pi)$ is a colouring $c$ of the pairs in~$\pi$ such that $c(p)\neq c(q)$ whenever $G$ contains an edge joining a vertex in $p$ to a vertex in~$q$.
The terms \emph{pair-chromatic number} and \emph{$k$-pair-colourable} are defined as expected.

An (\emph{edge-})\emph{colouring} of a 2-complex~$C$ is a colouring $c$ of the edges $e$ of~$C$ such that $c$ induces a pair-colouring of the link graph of~$C$ (which colours every vertex $I\se [0,1]_e$ of $L(C)$ with the colour~$c(e)$).
The terms \emph{edge-chromatic number} and \emph{$k$-edge-colourable} are also defined as expected.

\section{Proof of \texorpdfstring{\ref{main:sufficient}}{(1)}}

A \emph{2-pire map} is a paired graph $(G,\pi)$ where $G$ is planar.
We sometimes call $G$ a 2-pire map when $\pi$ is clear from context, or say that $G$ is a 2-pire map with pairing~$\pi$.
Isomorphisms between 2-pire maps are required to respect their pairings.
The name and definition are directly motivated by Heawood's \emph{$m$-pire problem}~\cite{heawood}, which was surveyed in~\cite{mgpire}.

\begin{example}\label{MbeddableGives2pire}
The link graphs of 2-complexes that embed in 3-manifolds, equipped with their default pairings, are examples of 2-pire maps.
\end{example}

\begin{lemma}[Heawood~\cite{heawood}; surveyed in~\cite{mgpire}]\label{Heawood2pire}
Every 2-pire map is 12-pair-colourable.
\end{lemma}

The \emph{paired quotient} of a paired graph $(G,\pi)$ is the graph $G/\pi$ obtained from~$G$ by identifying every two vertices that are paired by~$\pi$, keeping all edges.

\begin{lemma}\label{FromComplexToPairedLink}
Let $C$ be a 2-complex, and let $\pi$ denote the default pairing of the link graph~$L(C)$.
The following numbers are equal:
\begin{enumerate}[label={\textnormal{(\roman*)}}]
    \item the edge-chromatic number of the 2-complex~$C$,
    \item the pair-chromatic number of the link graph~$L(C)$, and
    \item the vertex-chromatic number of the paired quotient~$L(C)/\pi$.\qed
\end{enumerate}
\end{lemma}

\begin{proof}[Proof of \cref{main} \ref{main:sufficient}]
We combine \cref{MbeddableGives2pire} with \cref{Heawood2pire} and \cref{FromComplexToPairedLink}.
\end{proof}

\section{Proof of \texorpdfstring{\ref{main:necessary}}{(2)}}

\begin{lemma}\cite{heawood}\label{2pireMinimum}
There exists a 2-pire map whose pair-chromatic number is equal to~12.
\end{lemma}
\begin{proof}
For convenience, we have included \cref{fig:k12}, which shows a 2-pire map whose paired quotient is a~$K_{12}$, and whose pair-chromatic number is equal to~12 by \cref{FromComplexToPairedLink}.
\end{proof}

\begin{figure}[ht]
    \centering
    \includegraphics[width=.8\linewidth]{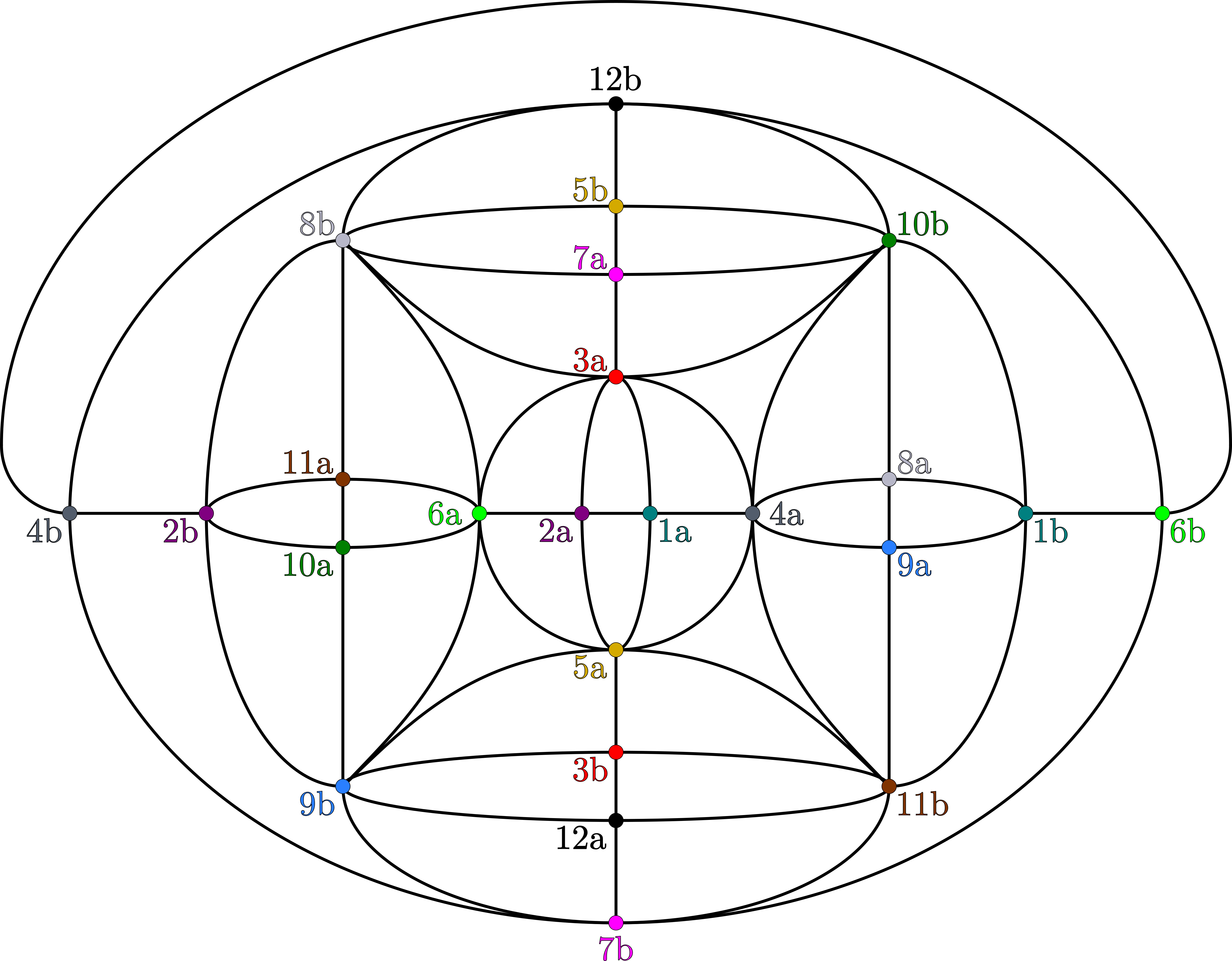}
    \caption{A 2-pire map whose pair-chromatic number is equal to 12. This was found by Kim Scott and published (in slightly different form) in~\cite{mgpire}}
    \label{fig:k12}
\end{figure}

To find a 2-complex $C$ that is not 11-colourable but embeds in~$\R^3$, it suffices by \cref{FromComplexToPairedLink} to construct $C$ so that its link graph with the default pairing contains a spanning copy of the 2-pire map as provided by \cref{2pireMinimum} while simultaneously making sure that $C$ embeds in~$\R^3$.
In the following, we will offer a 2-step construction that achieves just that.
The first step will be \cref{2pireToPunctured} below.

Recall that the degree of a vertex $v$ in a graph~$G$ is the number of edges of~$G$ that are incident with~$v$, counting loops twice.
The pairing $\pi$ of a paired graph $(G,\pi)$ is \emph{degree-faithful} if every two paired vertices have the same degree in~$G$.

We define \emph{punctured 2-complexes} as follows.
Let $G$ be a 1-complex, and let $F_i$ ($i\in I$) be pairwise disjoint copies of the closed strip $\Sp^1\times [0,1]$.
For each $i$, let $F_i^\ast$ denote the subspace of $F_i$ that corresponds to the circle $\Sp^1\times\{0\}$, and fix a continuous \emph{gluing map} $\varphi_i\colon F_i^*\to G$.
As for 2-complexes, we require the maps $\varphi_i$ to follow closed walks in~$G$ at constant speed.
The topological space $D$ obtained from $G$ and the closed strips $F_i$ by identifying $x$ with $\varphi_i(x)$ for all $i\in I$ and $x\in F_i^\ast$ is a \emph{punctured 2-complex}.
The name is motivated by the fact that every punctured 2-complex can be obtained from a genuine 2-complex by `puncturing' every 2-cell.
The \emph{link graph} of a punctured 2-complex is analogous to that of the link graph of a genuine 2-complex.
In fact, the link graph of a 2-complex is invariant under `puncturing'.
The subspaces of~$D$ obtained from the closed strips $F_i$ by gluing $F_i^\ast$ to $G$ are the \emph{punctured 2-cells} of~$D$.

The following definition is a variation of a similar definition in~\cite{WhitneyThree}.
Let $X$ be a set of points in~$\R^3$.
The \emph{shadow} of~$X$ is the set of all points in $\R^3$ that lie on a straight line segment between the origin and some point in~$X$.

\begin{lemma}\label{2pireToPunctured}
For every 2-pire map $G$ with a degree-faithful pairing~$\pi$ there exists a punctured 2-complex $C$ such that the link-graph of $C$ with the default pairing is isomorphic to $(G,\pi)$ and $C$ embeds in~$\R^3$.
\end{lemma}
\begin{proof}
Let $B_r$ denote the closed ball of radius~$r$ around the origin in~$\R^3$.
Since $G$ is planar, there is an embedding $\alpha$ of $G$ (viewed as a 1-complex) in the boundary of the unit ball~$B_1$.

Next, we construct a graph $H$ together with an embedding $\beta$ of $H$ in~$\R^3$, as follows.
The graph $H$ has only one vertex~$h$, which $\beta$ maps to the origin.
For every pair $p$ of vertices $u,v$ in the pairing $\pi$ of the 2-pire map~$G$, we add a loop $e_p$ to $H$ with end~$h$ and let $\beta$ map the interior of $e_p$ into $\R^3-h$ so that the intersection of $\beta(e_p)$ with $B_1$ is equal to the shadow of $\{\alpha(u),\alpha(v)\}$.
It is not hard to make sure that the images of distinct loops under $\beta$ do not intersect except in~$h$, for example as follows.
We enumerate the pairs in $\pi$ as $p_1,\ldots,p_n$.
Then we let $\beta$ map $e_p$ for $p=p_i=\{u,v\}$ to the union of the following three subspaces of~$\R^3$: the two straight line segments that link the origin to $\partial B_{i+1}$ and pass through $\alpha(u)$ and $\alpha(v)$, respectively, plus one of the obvious arcs that links $\alpha(u)$ and $\alpha(v)$ in the boundary~$\partial B_{i+1}$.

\begin{figure}[ht]
    \centering
    \includegraphics[height=18\baselineskip]{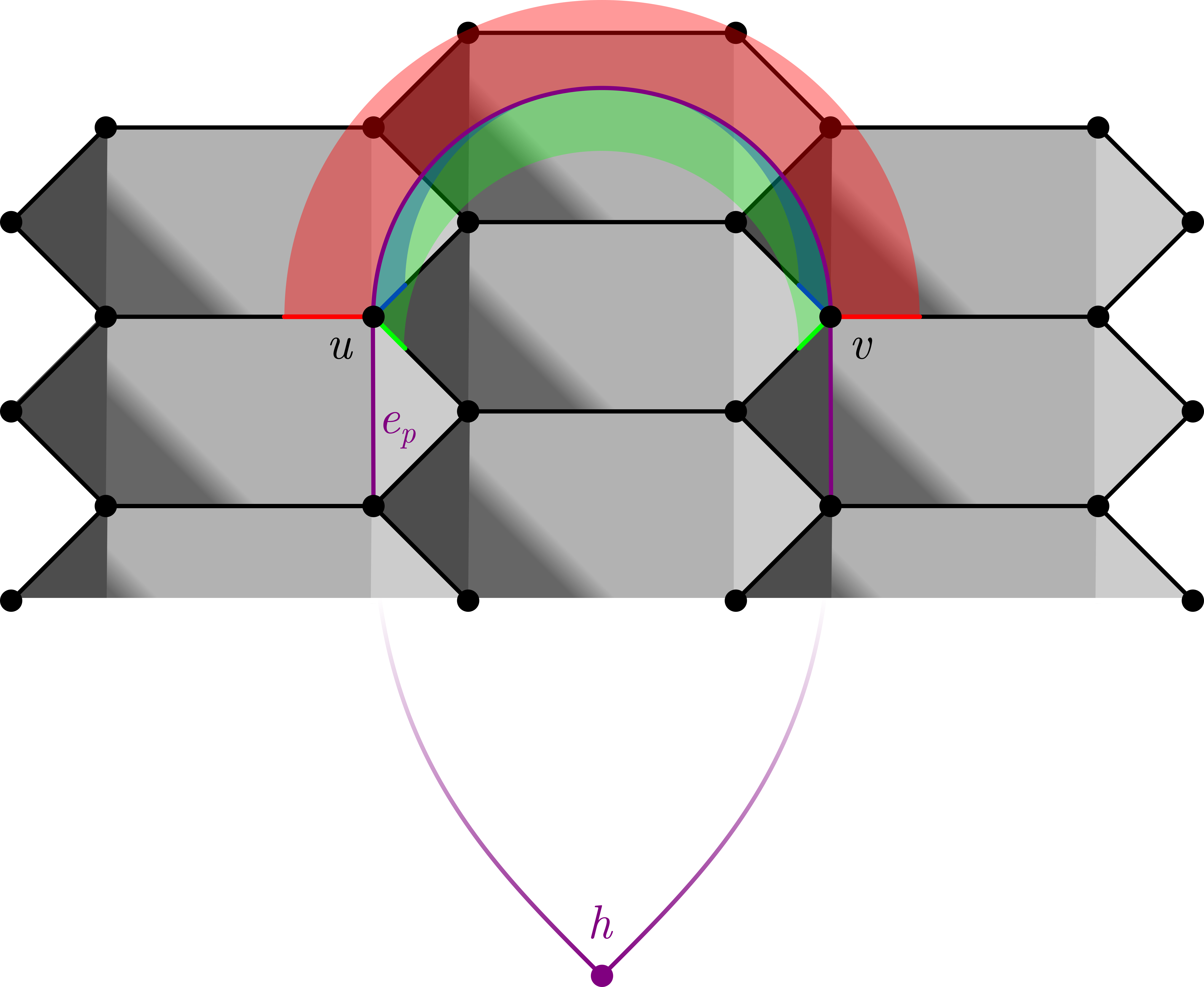}
    \caption{The image of $\iota_p$ for $p=\{u,v\}$. The black graph is $\alpha(G)$. The shadow of $\alpha(G)$ is indicated by the grey shaded areas.}
    \label{fig:attachment}
\end{figure}

The graph~$H$, viewed as a 1-complex, will be the 1-skeleton of the punctured 2-complex~$C$, which we construct next.
For every vertex $v$ of~$G$, let $S_v$ denote the image under $\alpha$ of the union of all third-edges of~$G$ that contain~$v$.
Note that the subspaces $S_v$ are pairwise disjoint, and that $S_u$ is homeomorphic to $S_v$ by a homeomorphism mapping $\alpha(u)$ to~$\alpha(v)$ for all pairs $\{u,v\}= p$ since $\pi$ is degree-faithful.
For each pair $p=\{u,v\}$ in~$\pi$, we informally link up $S_u$ and $S_v$ in $\R^3$ minus the interior of~$B_1$ by embedding the space $S_u\times [0,1]\cong S_v\times [0,1]$ so that this follows the topological path~$\beta(e_p)$, as shown in \cref{fig:attachment}.
More precisely, we find an embedding $\iota_p$ of $S_u\times [0,1]$ in $\R^3$ such that
\begin{itemize}
    \item $\iota_p$ maps $S_u\times \{0\}$ to~$S_u$ and $S_u\times \{1\}$ to~$S_v$;
    \item $\iota_p$ maps $\alpha(u)\times [0,1]$ to~$\beta(e_p)\setminus\mathring{B}_1$; and
    \item the image of $\iota_p$ avoids $B_1\sm (S_u\cup S_v)$.
\end{itemize}
We can greedily find the embeddings $\iota_p$ for all pairs $p\in\pi$ so that their images are pairwise disjoint: For example, if we construct $\beta$ using the balls of distinct radii as outlined above, we could even write down an explicit description of~$\iota_p$, which we do not as it would be tremendously tedious, but it is possible.

Let $C$ be the topological space obtained from the shadow of~$\alpha(G)$ by adding the images of the embeddings $\iota_p$ for all pairs~$p\in\pi$.
The construction of $C$ ensures that all connected components of $C\sm\beta(H)$ are homeomorphic to $\Sp^1\times (0,1]$.
Hence $C$ is a punctured 2-complex with 1-skeleton~$\beta(H)$.
By construction, the link graph of $C$ with its default pairing is homeomorphic to $G$ with the pairing~$\pi$.
\end{proof}

\begin{proof}[Proof of \cref{main}~\cref{main:necessary}]
By \cref{2pireMinimum}, there exists a 2-pire map $G$ with pairing~$\pi$ such that the pair-chromatic number of $G$ with regard to~$\pi$ is equal to~12.
For every edge of $G$ we add an edge in parallel, to make sure that all vertices of $G$ have even degree, which then allows us to add loops to $G$ so that $\pi$ becomes degree-faithful.

By \cref{2pireToPunctured}, there exists a punctured 2-complex $C$ as a subspace of $\R^3$ such that the link-graph of $C$ with the default pairing is isomorphic to $G$ with the pairing~$\pi$.
Let $F_1,\ldots,F_n$ be the punctured 2-cells of~$C$ and let $\varphi_1,\ldots,\varphi_n$ be the corresponding gluing maps.

For each $i=1,\ldots,n$ we do the following. 
Let $W_i$ denote the closed walk in~$H$ that $\varphi_i$ traverses at constant speed.
Let $U_i$ denote the smallest initial segment of~$W_i$ that uses an edge.
We define $W_i'$ to be the closed walk $W_i'=W_i U_i U_i^{-}W_i^{-}$, where $W^{-}$ denotes the reverse of a walk~$W$ and writing the walks in sequence means concatenation.

\begin{figure}[ht]
    \centering
    \includegraphics[width=.8\textwidth]{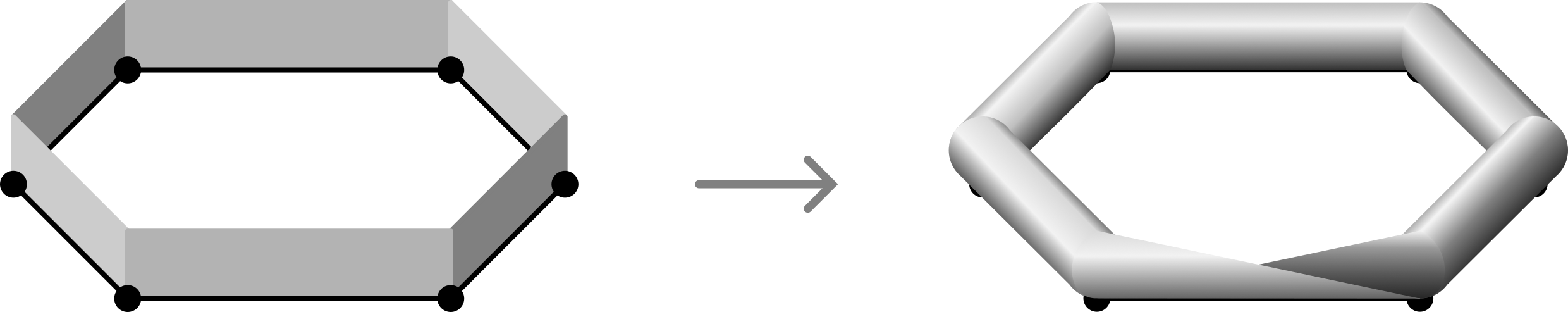}
    \caption{The replacement step}
    \label{fig:replacement}
\end{figure}

We obtain the 2-complex $C'$ from $C$ by replacing each punctured 2-cell $F_i$ with a genuine 2-cell $F_i'$ whose boundary we glue along $W_i'$.
By following $W_i'$ and working in close proximity to the punctured 2-cell $F_i\se\R^3$, we can embed the interiors of the $F_i'$ in~$\R^3$ as depicted in \cref{fig:replacement} so that we obtain an embedding of~$C'$ in~$\R^3$.
\end{proof}

\noindent\textbf{Acknowledgements.}
We are grateful to Noga Alon, Łukasz Bożyk and Michał Pilipczuk for drawing our attention to the construction of a 12-chromatic 2-pire map. We thank Johannes Carmesin for suggesting the problems in this note to us.
We thank Tsvetomir Mihaylov for interesting discussions~\cite{tsvetomir}.

\bibliographystyle{amsplain}
\bibliography{References}

\end{document}